\documentclass[12pt]{amsart}

\usepackage[OT2, T1]{fontenc}

\usepackage{amscd}
\usepackage{amsmath}
\usepackage{amssymb}
\usepackage{mathrsfs} 			
\usepackage{units}
\usepackage[all]{xy}

\usepackage{algorithm}
\usepackage{algorithmic}

\def\frk{\frak}               
\def\aa{{\frk a}}

\def\qq{{\frk q}}

\def\mm{{\frk m}}

\def\Phi{{\frk n}}
\def\Phi{{\frk N}}
\def\opn#1#2{\def#1{\operatorname{#2}}} 
\opn\chara{char} \opn\length{\ell} \opn\pd{pd} \opn\rk{rk}
\opn\projdim{proj\,dim} \opn\injdim{inj\,dim} \opn\rank{rank}
\opn\depth{depth} \opn\sdepth{sdepth} \opn\fdepth{fdepth}
\opn\grade{grade} \opn\height{height} \opn\embdim{emb\,dim}
\opn\codim{codim}  \opn\min{min} \opn\max{max}

\opn\Tr{Tr} \opn\bigrank{big\,rank}
\opn\superheight{superheight}\opn\lcm{lcm}
\opn\trdeg{tr\,deg}
\opn\reg{reg} \opn\lreg{lreg} \opn\ini{in} \opn\lpd{lpd}
\opn\size{size}
\opn\div{div} \opn\Div{Div} \opn\cl{cl} \opn\Cl{Cl}
\opn\Spec{Spec} \opn\Supp{Supp} \opn\supp{supp} \opn\Sing{Sing}
\opn\Ass{Ass} \opn\Min{Min}
\opn\Ann{Ann} \opn\Rad{Rad} \opn\Soc{Soc}
\opn\Im{Im} \opn\Ker{Ker} \opn\Coker{Coker} \opn\Am{Am}
\opn\Hom{Hom} \opn\Tor{Tor} \opn\Ext{Ext} \opn\End{End}
\opn\Aut{Aut} \opn\id{id}  \opn\deg{deg}

\opn\nat{nat}
\opn\pff{pf}
\opn\Pf{Pf} \opn\GL{GL} \opn\SL{SL} \opn\mod{mod} \opn\ord{ord}
\opn\Gin{Gin} \opn\Hilb{Hilb}
\opn\aff{aff} \opn\con{conv} \opn\relint{relint} \opn\st{st}
\opn\lk{lk} \opn\cn{cn} \opn\core{core} \opn\vol{vol}
\opn\link{link} \opn\star{star}
\opn\gr{gr}

\def\pot#1#2{#1[\kern-0.28ex[#2]\kern-0.28ex]}

\opn\dirlim{\underrightarrow{\lim}}
\opn\inivlim{\underleftarrow{\lim}}
\let\to=\rightarrow

\def\Implies{\ifmmode\Longrightarrow \else
        \unskip${}\Longrightarrow{}$\ignorespaces\fi}
\def\implies{\ifmmode\Rightarrow \else
        \unskip${}\Rightarrow{}$\ignorespaces\fi}
\def\iff{\ifmmode\Longleftrightarrow \else
        \unskip${}\Longleftrightarrow{}$\ignorespaces\fi}

\let\:=\colon
\newtheorem{Theorem}{Theorem}[]
\newtheorem{Lemma}[Theorem]{Lemma}
\newtheorem{Corollary}[Theorem]{Corollary}
\newtheorem{Proposition}[Theorem]{Proposition}
\newtheorem{Remark}[Theorem]{Remark}

\newtheorem{Example}[Theorem]{Example}

\newtheorem{Conjecture}[Theorem]{Conjecture}
\newtheorem{Question}[Theorem]{Question}

\let\epsilon\varepsilon
\let\phi=\varphi
\let\kappa=\varkappa
\def\qed{\ifhmode\textqed\fi
      \ifmmode\ifinner\quad\qedsymbol\else\dispqed\fi\fi}
\def\textqed{\unskip\nobreak\penalty50
       \hskip2em\hbox{}\nobreak\hfil\qedsymbol
       \parfillskip=0pt \finalhyphendemerits=0}
\def\dispqed{\rlap{\qquad\qedsymbol}}

\opn\dis{dis}
\def\pnt{{\raise0.5mm\hbox{\large\bf.}}}

\opn\Lex{Lex}

\begin{document}

\title{ On a question of Swan \\ {\tiny with an Appendix by K\k{e}stutis \v{C}esnavi\v{c}ius}}

\author{ Dorin Popescu }
\thanks{}

\dedicatory{In the memory of Hideyuki Matsumura.}

\address{Dorin Popescu, Simion Stoilow Institute of Mathematics of the Romanian Academy, Research unit 5,
University of Bucharest, P.O.Box 1-764, Bucharest 014700, Romania}
\email{dorin.popescu@imar.ro}

\begin{abstract} We show that a regular local ring is a filtered inductive limit of regular local rings, essentially of finite type over $\bf Z$.

\end{abstract}

\maketitle

\vskip 0.5 cm

\section*{Introduction}

Let $R$ be a regular ring. The following is a well known question concerning finitely generated projective modules over polynomial $R$-algebras.

\begin{Conjecture} (Bass-Quillen Conjecture, \cite[Problem IX]{B}, \cite{Q}) Every finitely generated projective module $P$ over a polynomial $R$-algebra, $R[T]$, $T=(T_1,\ldots, T_n)$ is extended from $R$, i.e. 
$P\cong R[T]\otimes_R(P/(T)P)$. 
\end{Conjecture}

The Bass-Quillen Conjecture (shortly the BQ Conjecture) has positive answers in the following cases

i) if $\dim R\leq 1$ by Quillen and Suslin (see \cite{Q}, \cite{Su}),

ii) if $R$ is essentially of finite type over a field by H. Lindel (see \cite{L}),

iii) if $R$ is a local ring of unequal characteristic, essentially of finite type over $\bf Z$, let us say $(R,\mm,k)$ with $p=$char$\ k\not \in \mm^2$ by Swan  \cite{Mu}.

Swan  noticed that it will be useful for the general question  to have  a positive answer to the following  one.

\begin{Question} (Swan \cite{Mu}) A regular local ring is a filtered inductive limit of regular local rings, essentially of finite type over $\bf Z$.
\end{Question}

A partial positive answer is given below.

\begin{Theorem} (\cite[Theorem 3.1]{P3}) The Swan question has a positive answer for a regular local ring $(R,\mm,k)$ in the following cases
\begin{enumerate}
\item $p=\mbox{char}\ k\not \in \mm^2$,
\item $R$ contains a field,
\item $R$ is excellent Henselian.
\end{enumerate}
\end{Theorem}

Using the above theorem we showed using Lindel's and Swan's results that the BQ Conjecture holds for regular
local ring containing a field, or regular local rings $(R,\mm,k)$ with $p=$char$\ k \not \in \mm^2$ (see \cite[Theorem 4.1]{P3}). After almost 30 years, we noticed that complete positive answers of the above questions  are  missing and there are still people interested to have them.

The purpose of the present paper is to give a  complete positive answer to the Swan Question (see Theorem \ref{s} where important is the last sentence, which follows from   the first part of the statement necessary only for the proof). When $\dim R=1$ our Theorem \ref{t0} uses a condition of separability which we remove it in Theorem \ref{t2}.  Unfortunately, our result seems to be not very useful for the Bass-Quillen Conjecture (see Remark \ref{r1}). The appendix shows that our Corollary \ref{c}
could be used to reduce a well known conjecture on {\em purity} to the complete case (see Proposition \ref{pu}).

\vskip 0.3 cm

\section{Discrete Valuation Rings of unequal characteristic.}
\vskip 0.3 cm

Let $(A,\mm)$ be a discrete valuation ring (a  DVR for short) of unequal characteristic. Then $A$ dominates ${\bf Z}_{(p)}$, where $p=$char $k$, $k:=A/\mm$. We suppose that $p\in \mm^2$ and the field extension $k\supset {\bf F}_p$  is {\em separably generated}.

\begin{Remark}\label{r0} Suppose that the fraction field $K=\mbox{Fr}(A)$ of $A$ is a finite type extension of $\bf Q$.
Then it is possible that $k/{\bf F}_p$ is not a finite type field extension (see \cite[Theorem 6.1]{Z}). Thus 
$k/{\bf F}_p$ need not  be  separably generated.
\end{Remark}

\begin{Lemma} \label{lm}  There exists a DVR subring $B$  of $A$ and a regular parameter $x$ of $A$ such that

\begin{enumerate}
\item $B\subset A$ is a ramified extension inducing an algebraic separable  extension on the residue fields,
\item B contains  a power $b=x^e$, $1<e\in {\bf N}$ of    $x$, which is a regular parameter in $B$, 
\item $C=B[x]_{(x)}$ is a DVR subring of $A$ such that $C\cong (B[X]/(X^e-b))_{(X)}$,
\item the extension $C\subset A$ is unramified.
\end{enumerate}
\end{Lemma}  

\begin{proof} A lifting  of a separable transcendence base ${\bar y}=({\bar y}_i)_{i\in I}$ of $k$ over ${\bf F}_p$ induces a system of algebraically independent elements  $y:=( y_i)_{i\in I}$, $y_i\in A$ of $K$  over ${\bf Q}$.  
 Then the ring $C_0={\bf Z}[(Y_i)_{i\in I}]_{p{\bf Z}[(Y_i)_{i\in I}]}$ is a DVR (see e.g \cite[Theorem 83]{M}).  Consider the flat map  $\psi_0:C_0\to A$ given by $Y\to y$. Note that $\psi_0$ is  ramified and induces an algebraic separable extension on the residue fields. Suppose that $p=x^st$ for some regular parameter $x$ of $A$ and an unit element $t\in A$, $1< s\in {\bf N}$.

We claim that $t$ is not  transcendental over ${\bf Q}(y)$. Indeed, choose $r$ such that $p^r>s$.  Then  $B= B'\cap  {\bf Q}(y,t^{p^r})$ is a DVR  and the residue field extension induced by $B\subset B'$ is pure  inseparable because  $u^{p^r}\in B $ modulo $\mm\cap B$ for every  $u\in B'$.   As it is also   algebraic separable by hypothesis, we see that this residue field extension is trivial and so $B'$ must be  a ramified extension of $B$ of order $e_{B'/B}=p^r>s$ which is false.

 The DVR  $B'=A\cap {\bf Q}(y,t)$ contains $x^s=p/t$. If $B'$ is an unramified extension of $C_0$ then the polynomial $X^s-p/t$ is irreducible in $B'[X]$, 
 $$C=B'[x]_{(x)}\cong (B'[X]/(X^s-p/t))_{(X)}$$
  is a DVR subring of $A$, the extension $C\subset A$ is unramified and induces an algebraic separable extension on the residue fields.

Otherwise (this is possible as shows Example \ref{e}),  choose  a regular parameter $z'$ of $B'$. Then $z'=x'^{s'}t'$ for some regular parameter $x'$ of $A$, $1\leq s'<s$ and an invertible element $t'\in A$. Note that the field extensions ${\bf Q}(y)\subset {\bf Q}(y,t')\subset {\bf Q}(y,t)$ are finite and the last one of degree $s'$ because ${\bf Q}(y,t')={\bf Q}(y,s')$, ${\bf Q}(y,t)={\bf Q}(y,s)$. 

If $s'=1$ then $B'\subset A$ is unramified. Using again this procedure we  arrive in some steps  to a DVR subring $B''$ of $ B'$ such that the extension $B''\subset B'$ is ramified. This is because the degree of the corresponding fraction field extensions over ${\bf Q}(y)$ decreases in each step, in the worst  case $B''=C_0$.

 Thus we may assume that $s'>1$. 
 Repeating this procedure for $A\cap {\bf Q}(y,t,t')$ and so on,
we arrive in some steps to a DVR  subring  $(B,(z))$  of $A$ containing a power $b=x^e$, $e>1$ of a regular parameter $x$ of $A$ with $z\not \in x^{e'}A$  for $e'<e$. Then $C=B[x]_{(x)}\cong (B[X]/(X^e-b))_{(X)}$ is a DVR subring of $A$, the extension $C\subset A$ is unramified and induces an algebraic separable extension on the residue fields. 
\hfill\ \end{proof} 

\begin{Example}\label{e} {\em Note that  $B=({\bf Z}[Y]/(Y^2-5))_{(Y)}$ is a DVR, ramified extension of $C={\bf Z}_{(5)}$. The polynomial $f=X^4-5/(1+Y)\in B[X] $ is irreducible and $D=(B[X]/(f))_{(X)}$ is  a DVR, ramified extension of $B$. Note that $D\cap {\bf Q}(Y)=B$ is a ramified extension of $C$. Thus in the above proof it is possible that  $B'$ could be indeed a ramified extension of $C_0$.}
\end{Example}     

Using by recurrence the above lemma we get the following proposition.

\begin{Proposition} \label{pm} In the notations of the above lemma, there exists  extensions of DVR subrings of $A$,
 $$C_0\subset B_1\subset C_1\subset \ldots B_r\subset C_r\subset B_{r+1}=A,$$ $C_0$ being defined in Lemma \ref{lm}, such that 
\begin{enumerate}
\item $C_i\subset B_{i+1}$ is an unramified extension for any $0\leq i\leq r$ inducing an algebrac separable  extension on the residue fields,
\item  $C_i\cong (B_i[X]/(X^{e_i}-b_i))_{(X)}$ for some $b_i\in B_i$, $1<e_i\in {\bf N}$, which is a regular parameter in $B_i$.
\end{enumerate}
\end{Proposition}

\begin{Theorem} \label{t0} Suppose that $K=$Fr$A$ is  a field  extension of $\bf Q$ not necessarily of finite type. Then $A$ is a filtered inductive union of    regular local subrings $(R_i)_{i\in I}$ of $A$, essentially of finite type over $\bf Z$.
\end{Theorem}
\begin{proof}    It is enough  to show that for a finite type $\bf Z$-algebra $E\subset A$ there exists a regular local subring $R\subset A$ which contains $E$ and it is essentially of finite type over $\bf Z$. 
 By  Proposition \ref{pm}  there exists  extensions of DVR subrings of $A$,
  $$C_0\subset B_1\subset C_1\subset \ldots B_r\subset C_r\subset B_{r+1}=A$$
   such that (1), (2) hold above.
Apply the classical N\'eron Desingularization (see \cite{N}, \cite[Theorem 1]{KAP}) for the case $C_r\subset B_{r+1}=A$. Then there exists a regular local subring $R_r\subset A$, which contains $C_r[E]$ 
and it is essentially of finite type over $C_r$. If $r=0$ then $A$ is an unramified extension of $C_0$ inducing an algebraic separable extension on residue fields and  we are done.

Suppose that $r>0$. Using again the  N\'eron Desingularization for the extension $C_{r-1}\subset B_r$ we see that $B_r$ is a filtered inductive union of regular local subrings essentially of finite type over $C_{r-1}$.
Then there exists a regular local subring $T_{r-1}$ of $B_r$ essentially of finite type over $C_{r-1}$ such that
\begin{enumerate}

\item $b_r$ belongs to a regular system of parameters $z_{r-1}$ of $T_{r-1}$ and there exist a regular local subring $T'_{r-1}$ of $B_r$ isomorphic to $(T_{r-1}[X]/(X^{e_{r-1}}-b_{r-1}))_{(z_{r-1},X)}$,

 \item $R_r$ is defined over  $T'_{r-1}$, that is there exists  a regular local subring $R_{r-1}\supset T'_{r-1}$ of $B_r$ such that $R_r\cong R_{r-1}\otimes_{T'_{r-1}} C_r$. 
\end{enumerate} 

 Applying by recurrence this procedure we find a regular local subring $R_0$ of $A$, which is essentially of finite type over $C_0$ and contains $C_0[E]$. This is enough.      
\hfill\  \end{proof}

\begin{Remark}{\em If $p\not\in \mm^2$ then the problem is easier (see \cite[Theorem 3.1]{P3}).}
\end{Remark}

\vskip 0.3 cm

\section{Regular local rings of unequal characteristic}

\vskip 0.3 cm

Let $(R,\mm,k)$ be a regular local ring of dimension $n$ and  $p=$char $k$. We suppose that $0\not =p\in \mm^2$.

\begin{Lemma} \label{l1} There exists a regular system of parameters $x=(x_1,\ldots,x_n)$ for $R$ such that
$(p,x_2,\ldots,x_n)$ is a system of parameters for $R$. For any such, the map  ${\bf Z}[X_2,\ldots,X_n]_{(p,X_2,\ldots,X_n)}\to R$ given by $X_i\to x_i$, $1<i\leq n$ is flat and induces a ramified extension of DVRs modulo $(x_2,\ldots,x_n)$.
\end{Lemma}

\begin{proof} The second statement follows from the first and the flatness criterion (see e.g. \cite[20.C]{M})
Suppose that $n>1$ and let $z=(z_1,\ldots,z_n) $ be a regular system of parameters of $R$.
Using induction on $n$ it is enough to choose $x_n$ from the infinite set $z_n+\mm^2$ which does not divide $p$.
\hfill\ \end{proof}    

For the next results we need some preparations.

A ring morphism $u:A\to A'$ of Noetherian rings has  {\em regular fibers} if for all prime ideals $p\in \Spec A$ the ring $A'/pA'$ is a regular  ring.
It has {\em geometrically regular fibers}  if for all prime ideals $p\in \Spec A$ and all finite field extensions $K$ of the fraction field of $A/p$ the ring  $K\otimes_{A/p} A'/pA'$ is regular.
A flat morphism of Noetherian rings $u$ is {\em regular} if its fibers are geometrically regular. If $u$ is regular of finite type then $u$ is called {\em smooth}.

The following theorem extends N\'eron's desingularization (see \cite{N}, \cite{KAP}) and it was useful to solve different problems concerning the projective modules over regular rings, or from the Artin Approximation Theory (see
 \cite{P}, \cite{P1},  \cite{P3}, \cite{P4}, \cite{S}).

\begin{Theorem} (General N\'eron Desingularization, Popescu \cite{P0}, \cite{P}, \cite{P'}, \cite{P1}, \cite{P2}, Andr\'e \cite{An}, Swan \cite{S})\label{gnd}  Let $u:A\to A'$ be a  regular morphism of Noetherian rings and $B$ an  $A$-algebra of finite type. Then  any $A$-morphism $v:B\to A'$   factors through a smooth $A$-algebra $C$, that is, $v$ is a composite $A$-morphism $B\to C\to A'$.
\end{Theorem}

  Let $A$ be a Noetherian ring, $E=A[Y]/I$, $Y=(Y_1,\ldots,Y_q)$. If $f=(f_1,\ldots,f_r)$, $r\leq q$ is a system of polynomials from $I$ then we can define the ideal $\Delta_f$ generated by all $r\times r$-minors of the Jacobian matrix $(\partial f_i/\partial Y_j)$.   After Elkik \cite{El} let $H_{E/A}$ be the radical of the ideal $\sum_f ((f):I)\Delta_fB$, where the sum is taken over all systems of polynomials $f$ from $I$ with $r\leq q$.
 $H_{E/A}$ defines the non smooth locus of $E$ over $A$.

\begin{Proposition} \label{p1} In the notation and hypotheses of Lemma \ref{l1}, let $E\subset R$ be a 
$C:=({\bf Z}[x_2,\ldots,x_n])_{(p,x_2,\ldots,x_n)}\cong ({\bf Z}[X_2,\ldots,X_n])_{(p,X_2,\ldots,X_n)}$-subalgebra of finite type. Suppose that $n>1$. Then the inclusion $v:E\to R$ factors through a finite type $C$-algebra $F$, let us say $v$ is the composite map $E\to F\xrightarrow{w} R$ such that $w(H_{F/C})R$ contains a power of $p$.
\end{Proposition}    
\begin{proof} 
Let $q$ be a minimal prime ideal of $h_E=\sqrt{v(H_{E/C})R}$ which does not contain $p$. Then $C\to R_q$ is regular and using Theorem \ref{gnd}, or \cite[Lemma 8]{P2} we see that
 $v$ factors through a finite type $C$-algebra $F_1$,  let us say $v$ is the composite map $E\to F\xrightarrow{w_1} R$,  such that $h_{F_1}=\sqrt{w_1(H_{F_1/C})R}$    strictly contains $q$. Note that $w_1$ is not necessarily  injective. 
 Step by step we arrive in this way to some  new $F_r$ and $w_r:F_r\to R$ such that all minimal prime ideals 
 of   $h_{F_r}:=\sqrt{w_r(H_{F_r/C})R}$ contain $p$, that is, $p\in h_{F_r}$. We are done.
\hfill\ \end{proof}
 
We will need the next result, which is in fact \cite[Proposition  5]{P2} (see also \cite[Proposition 3]{ZKPP}) written in our special case.
 \begin{Proposition}\label{p2}
 Let $A$ and $A'$ be Noetherian  rings  and $u:A\to A'$ be a ring morphism. Suppose that $A'$ is local and let $E=A[Y]/I$, $Y=(Y_1,\ldots,Y_s)$, $f=(f_1,\ldots,f_r)$, $r\leq s$ be a system of polynomials from $I$, $(M_j)_{j\in [l]}$  some $r\times r$-minors    of the Jacobian matrix $(\partial f_i/\partial Y_{j'})$,  $(N_j)_{j \in [l]} \in  ((f):I)$ and set $P:=\sum _{j=1}^l N_jM_j$. Let  $v:E\to A'$ be an $A$-morphism. Suppose that
\begin{enumerate}
 \item{} there exists a non zero divisor  $d\in A$ such that $d\equiv P $ modulo $I$, which is also non zero divisor in $A'$, and

 \item{} there exist an  $A$-algebra $D$ of finite type and an $A$-morphism $\omega:D\to A'$ such that $d$ is a non zero divisor in $D$, $\Im v\subset \Im\omega +d^3A'$ and for ${\bar A}=A/(d^3)$  the map $\bar{v}={\bar A} {\otimes}_{A} v: \bar{E}=E/d^3E \to \bar{A}'=A'/d^3A'$ factors through $\bar{D}=D/d^3D$.
\end{enumerate}
 Then there exists a smooth $D$-algebra $D'$  such that $v$ factors through $D'$, let us say $v$ is the composite map $D\to D'\xrightarrow{w} A'$ and  $h_D=\sqrt{\omega(H_{D/A})A'}\subset h_{D'}=\sqrt{w(H_{D'/A})A'}$.
 \end{Proposition}

\begin{Remark}\label{r} {\em Actually the proof from  \cite[Proposition  5]{P2} asks also for $D$  and $A'$ to be flat over $A$. In our case, it is only necessary to assume that $d$ is a non zero divisor in $D$ and $A'$. We can even assume that $d$ is not regular in $A'$, but in this case we should change $d^3$ by $d^{3c}$ for some $c$ with $(0:_{A'}d^c)=(0:_{A'}d^{c+1})$.}
\end{Remark}

Next theorem extends Theorem \ref{t0} in the case when $k$ is not separable generated over ${\bf F}_p$.

\begin{Theorem} \label{t2} Let $(A,\mm,k)$ be a Noetherian local ring, $s=(s_1,\ldots,s_m)$ some positive integers and $\gamma=(\gamma_1,\ldots,\gamma_m)$ a system of nilpotents of $A$. Suppose that $0\not =p\in \mm^2$, $R=A/(\gamma)$ is a DVR and  $A$ is a flat  $\Phi={\bf Z}_{(p)}[\Gamma]/(\Gamma^s)$-algebra, $\Gamma\to \gamma$ with $\Gamma=(\Gamma_1,\ldots,\Gamma_m)$ some  variables, and $(\Gamma^s)$ denotes the ideal $(\Gamma_1^{s_1},\ldots, \Gamma_m^{s_m})$. Then $A$ is a  filtered inductive limit of some  Noetherian local $\Phi$-algebras $(A_i)_i$ essentially of finite type with $A_i/\Gamma A_i$ regular local rings. In particular, a DVR   is a filtered inductive limit of   some  regular local rings, essentially of finite type over $\bf Z$.
\end{Theorem}
\begin{proof} 
 Let $x\in A$ be an element inducing a local parameter of $R$,  $E\subset A$ be  a finite type $\Phi$-algebra  and  $v:E\to A$  be the inclusion. Actually, we could simply take a morphism (not necessarily injective) $v:E\to A $ for some finite type $\Phi$-algebra $E$.
We may assume  that $p\equiv x^et_0$ modulo $(\gamma)$ for some $e\in {\bf N}$ and  $t_0\in A\setminus \mm$. Moreover, we may suppose that $p= x^et_0+\sum_{j=1}^m t_j\gamma_j$ for some 
$t_j\in A$, $j\in [m]$.  As in Proposition \ref{p1} we may assume  that $v(H_{E/\Phi})A$ contains a power  $b$ of $p$ because the map $\Phi\to A_q$ is a regular map  for all $q\in \Spec A$ with  $p\not \in q$. 

Following the proof of \cite[Theorem 2]{ZKPP},  we assume that
$b=\sum_{i=1}^qv(a_i)z_i$, where $z_i\in A$, $a_i\in H_{E/\Phi}$. Set $E_0=E[Z]/(f) $, where $ f=-b +\sum_{i=1}^qa_iZ_i\in E[Z]$, $Z=(Z_1,\ldots,Z_q)$,
and let $v_0:E_0\to A$ be the map of $E$-algebras given by $Z\to z$. Changing $E$ by $E_0$ we may assume that $b\in H_{E/\Phi}$.
 
Let $E\cong \Phi[Y]/I$, $Y=(Y_1,\ldots Y_m)$. Using \cite[Lemma 4]{ZKPP} we may change again $E$ to assume that a power $d$ of $b$ is in  $((f):I)\Delta_f$ for some $f=(f_1,\ldots,f_r)$, $r\leq s$ from $I$, where $\Delta_f$ denotes the ideal generated by the $r\times r$-minors of $(\partial f/\partial Y)$. Then $d$  has the form $d\equiv P= \sum_{i=1}^qM_iL_i\ \mbox{modulo}\ I$,
for some $r\times r$ minors $M_i$
of $(\partial f/\partial Y) $ and $L_i\in ((f):I)$.

 Let $(\Lambda,(p,\Gamma),k)$ be the (unique) Cohen $\Phi$-algebra with residue field $k$, that is  $\Lambda$ is complete  and the map $\Phi\to \Lambda$ is flat (even regular) (see e.g.  \cite[Theorem 83]{M}). $\Lambda$ is a filtered inductive limit of some Noetherian local $\Phi$-algebras $(C_j)_j$ essentially smooth (see Theorem \ref{gnd}). Then the completion of the DVR $A/(\gamma)$ is a finite free module over the DVR $\Lambda/\Gamma \Lambda$ with the base $\{1,x,\ldots, x^{e-1}\}$ and so $A/(\gamma,d^3)$ is a finite free module over $\Lambda/(\Gamma,d^3)$ with the same base.
Using \cite[(20.G)]{M} applied to $\Phi\to \Lambda\to A$ we get $A$ flat over $\Lambda$. 
   It follows that $A/(d^3)$ is a finite free module over $\Lambda/(d^3)$ with the same base and $t_j\equiv \sum_{i=0}^{e-1}t_{ji}x^i$ modulo $d^3$ for some $t_{ji}\in \Lambda$. Thus there exists a smooth ${\Phi}$-algebra $ C$ such that $ (t_{ji})_{ji}$ are contained in the image of the canonical map $ \psi: C\to \Lambda$, let us say $t_{ji}$ are the images of some $t'_{ji}\in C$. Let $\phi$ be the composite  map
    $$C[X]/(d^3)\to \Lambda[X]/(d^3)\to {\tilde A}=A/(d^3)$$
     induced by $\psi$ and $X\to x$.  Actually, $\tilde A$ is a filtered inductive limit of $\Phi$-algebras of type ${\tilde B}= C[X]/(d^3, p-t'_0X^e-\sum_{j=1}^mt'_j\Gamma_j)$, where $t'_j=\sum_{i=0}^{e-1}t'_{ji}X^i$ and $C$ is smooth over $\Phi$ and containing $(t'_{ji})$.

 We may assume that  $ C=({\Phi}[U]/( g))_{a g'}$, $U=
(U_1,\ldots,U_l)$ for a monic polynomial $ g$ (in $U_1$), $ a\in {\Phi}[U]$,  where  ${\bar g}'= \partial{\bar g}/\partial U_1$ (see e.g \cite[Theorem 2.5]{S}).    We may take $C$ such that the map ${\tilde v}:{\tilde E}:=E/d^3E\to {\tilde A}$ induced by $v$, factors through $\tilde B$, let us say $\tilde v$ is the composite map $ {\tilde E}\to {\tilde B}\xrightarrow{\tilde w} {\tilde A}$, the last map being the limit map.

 Choose a lifting  $u=(u_i)_{1\leq i\leq l}$ in $A$ of $({\tilde w}(U_i))_{1\leq i\leq l}$. We have $g(u)=d^3z$ for some $z\in A$. Set $D_1=C[X,U,Z]/(g-d^3Z)$ and let $w:D_1\to A$  be given by $(X,U,Z)\to (x,u,z)$. Clearly,   $D_2=(D_1/(\Gamma))_{w^{-1}((x))}$ is a regular local ring being essentially smooth over a polynomial ring over $\bf Z$ and $p,X$ are contained in a regular system of parameters of $D_2$. Then $D_2/(p-t_0'X^e)$ is a regular local ring and denote the map 
 $$D=D_1/(p-t_0'X^e-\sum_{j=1}^mt'_j\Gamma_j)\to A$$
  induced by $w$ also by $w$.

 We have $\Im v\subset \Im w + d^3A$ and $v$ factors modulo $d^3$ through $w$. Since $d$ is a non zero divisor in  $D$, $A$ we may apply 
 Proposition \ref{p2} (see also Remark \ref{r}) for $\Phi\to A$, $E$,   $D$ and $\omega=w$. We get an essentially smooth $D$-algebra $D'$   such that $v$ factors through $D'$. Since $D/\Gamma D$ is regular local  we get $D'/\Gamma D'$ regular too.  This is enough using e.g. \cite[Lemma 1.5]{S} (see below), $\mathcal S$ being the set of the regular local $\Phi$-algebras.  
\hfill\ \end{proof}

\begin{Lemma}
Let $\mathcal S$ be a class of finitely presented $\Phi$-algebras. Let  $A$ be a $\Phi$-algebra. Then the following statements are equivalent: 

(1) $A$ is a filtered inductive limit of algebras from $\mathcal S$,

(2) If $E$ is a  finitely presented $\Phi$-algebra and  $v:E\to A$ is a morphism  
then $v$ factors through an algebra from $\mathcal S$.
\end{Lemma}\

The following theorem is a positive  answer  of a question of Swan \cite{Mu}.

\begin{Theorem} \label{s}  Let $(A,\mm,k)$ be a Noetherian local ring, $s=(s_1,\ldots,s_m)$ some positive integers and $\gamma=(\gamma_1,\ldots,\gamma_m)$ a system of nilpotents of $A$. Suppose that $0\not =p\in \mm^2$, $R=A/(\gamma)$ is a regular local ring and  $A$ is a flat  $\Phi={\bf Z}_{(p)}[\Gamma]/(\Gamma^s)$-algebra, $\Gamma\to \gamma$ with $\Gamma=(\Gamma_1,\ldots,\Gamma_m)$ some  variables, and $(\Gamma^s)$ denotes the ideal $(\Gamma_1^{s_1}\cdots \Gamma_m^{s_m})$.
 Then $A$ is a   filtered inductive limit of some  Noetherian local $\Phi$-algebras $(F_i)_i$ essentially of finite type with $F_i/\Gamma F_i$ regular local rings. In particular, a regular local ring $(R,\mm,k)$ of unequal characteristic  is a filtered inductive limit of regular local rings essentially of finite type over $\bf Z$.
\end{Theorem}
\begin{proof}  Let  $x_1,\ldots,x_n\in A$, $n=\dim A$ be a system of elements defining a regular system of parameters in R, $E\subset A$   a finite type $\Phi$-algebra  and  $v:E\to A$  be the inclusion. Suppose that $n>1$ by Theorem \ref{t2} and let $p\equiv t_0\Pi_{i=1}^{l_1}z_{1i}^{\alpha_{1i}}$ modulo $(\gamma)$ for some $\alpha_{1i}\in {\bf N}$, $t_0\in A\setminus \mm$ and some  elements $z_{1i}$ of $A$ inducing irreducible elements in $R$, with $(z_{1i},\gamma)\not =(z_{1i'},\gamma)$ for $i\not =i'$, $R$ being a unique factorization domain.  Moreover, we may suppose that $p=  t_0\Pi_{i=1}^{l_1}z_{1i}^{\alpha_{1i}}+\sum_{j=1}^m t_j\gamma_j$ for some 
$t_j\in A$, $j\in [m]$.

As in Proposition \ref{p1} we may assume  that $v(H_{E/\Phi})A$ contains a power  of $p$ and as in Theorem \ref{t2} we may assume that $E\cong \Phi[Y]/I$, $Y=(Y_1,\ldots Y_m)$ and  that a power $d$ of $p$ is in  $((f):I)\Delta_f$ for some $f=(f_1,\ldots,f_r)$, $r\leq s$ from $I$. Then $d$  has the form $d\equiv P= \sum_{i=1}^qM_iL_i\ \mbox{modulo}\ I$,
for some $r\times r$ minors $M_i$
of $(\partial f/\partial Y) $ and $L_i\in ((f):I)$.

Consider as in Theorem \ref{t2}
 the  Cohen $\Phi$-algebra $(\Lambda_{1i} ,(p,\Gamma), \mbox{Fr}(A/(z_{1i},\gamma)))$ with residue field  Fr$(A/(z_{1i},\gamma))$, that is  $\Lambda_{1i}$ is complete  and the map $\Phi\to \Lambda_{1i}$ is flat (even regular). $\Lambda_{1i}$ is a filtered inductive limit of some Noetherian local $\Phi$-algebras $(C_{ij})_j$ essentially smooth (see Theorem \ref{gnd}). Then the completion of the DVR  $R_{(z_{1i})}$ is a finite free module over the DVR $\Lambda_{1i}/\Gamma \Lambda_{1i}$ with the base $\{1,z_{1i},\ldots, z_{1i}^{\alpha_{1i}-1}\}$ and so $(A/(\gamma,d^3))_{(z_{1i})}$ is a finite free module over $\Lambda_{1i}/(\Gamma,d^3)$ with the same base.  It follows that $(A/(d^3))_{(z_{1i},\gamma)}$ is a finite free module over $\Lambda_{1i}/(d^3)$ with the same base. 
 
  Set $S=A\setminus \cup_{i=1}^{l_i}(z_{1i},\gamma)A$ and ${\tilde A}:=A/(d^3)$, ${\tilde{\Phi}}:={\Phi}/(d^3)$. Then $S^{-1}{\tilde A} $ is a product of the local Artinian rings ${\tilde A}_{(z_{1i},\gamma)}$ which are finite free modules over $({\tilde \Lambda}_{1i})=\Lambda_{1i}/(d^3)$, the base of $S^{-1}{\tilde A} $ over ${\tilde \Lambda}_1:= \Pi_{i=1}^{l_1} {\tilde \Lambda}_{1i}$ is given by 
$$\{z_{11}^{\epsilon_{11}}\cdots 
z_{1l_1}^{ \epsilon_{1l_1}}: 0\leq \epsilon_{11}<\alpha_{11},\ldots, 0\leq \epsilon_{1l_1}<\alpha_{1l_1}\}.$$ 
Let $t_j\equiv \sum t_{j\epsilon_{11}\ldots \epsilon_{1l_1}} z_{11}^{\epsilon_{11}}\cdots z_{1l_1}^{\epsilon_{1l_1}}$ modulo $(d^3)$ for some $t_{j\epsilon_{11}\ldots \epsilon_{1l_1}}\in {\tilde \Lambda}_1$.

Then we may choose some smooth  $\tilde{\Phi}$-algebras of type $({\tilde C}_{1i})_i$ such that the limit maps give a map ${\tilde \psi}:{\tilde C}_1:=\Pi_{i=1}^{l_1}{\tilde C}_{1i}\to {\tilde \Lambda}_1$ whose image contains $( t_{j\epsilon_{11}\ldots \epsilon_{1l_1}})$, let us say  $ t_{j\epsilon_{11}\ldots \epsilon_{1l_1}}={\tilde \psi}(  t'_{j\epsilon_{11}\ldots \epsilon_{1l_1}}) $
for some 
 $ t'_{j\epsilon_{11}\ldots \epsilon_{1l_1}}\in {\tilde C}_1$.
 Let ${\tilde\phi}_1:{\tilde C}_1[Z_{11},\ldots,Z_{1l_1}]\to S^{-1}{\tilde A}$ for some new variables $Z_{11},\ldots,Z_{1l_1}$ be the map induced by $\tilde{\psi}$ and $Z_{1i}\to z_{1i}$. 
Note that 
$$t'_j=\sum t'_{j\epsilon_{11}\ldots \epsilon_{1l_1}} Z_{11}^{\epsilon_{11}}\cdots Z_{1l_1}^{\epsilon_{1l_1}}\in {\tilde C}_1[Z_{11},\ldots,Z_{1l_1}]$$
 is mapped to $t_j$ modulo $d^3$ and 
 $S^{-1}({\tilde A})$ is a filtered inductive limit of localizations of $\tilde{\Phi}$-algebras of finite type
  $${\tilde B}_1:={\tilde C}_1[Z_{11},\ldots,Z_{1l_1}]/( p-t_0'Z_{11}^{\alpha_{11}}\cdots Z_{1l_1}^{\alpha_{1l_1}}-\sum_{j\in [m]}t'_j\Gamma_j).$$
   Thus the composite map $E\xrightarrow{v} A\to S^{-1}{\tilde A}$ factors through  such an algebra ${\tilde B}_1$.
  
   Clearly,
${\tilde C}_1$ is a localization of a finite type $\tilde{\Phi}$-algebra ${\tilde C}$ with ${\tilde \phi}_1({\tilde C})\subset {\tilde A}$ and ${\tilde \phi}_1(H_{{\tilde C}/{\tilde{\Phi}}})\not\subset \cup_{i=1}^{l_1}(z_{1i},\gamma){\tilde A}$. Unfortunately we cannot assume that  $t_j'\in {\tilde C}[Z_{11},\ldots,Z_{1l_1}]$ but we may suppose that there exist
$c\in {\tilde C}$ and $t_j''\in {\tilde C}[Z_{11},\ldots,Z_{1l_1}]$ such that $ct_j'=t_j''$ and 
${\tilde \phi}_1(c)\in {\tilde A}\setminus  \cup_{i=1}^{l_i}(z_{1i},\gamma){\tilde A}$. It follows that 
${\tilde B}_1$ is a localization of 
$${\tilde B}={\tilde C}[T',Z_{11},\ldots,Z_{1l_1}]/(cT_0'-t_0'',\ldots,cT_m'-t_m'', p-T_0'Z_{11}^{\alpha_{11}}\cdots Z_{1l_1}^{\alpha_{1l_1}}-\sum_{j\in [m]}T'_j\Gamma_j),$$ $T'=(T'_0,\ldots,T'_m)$ being some new variables, and we may consider the map ${\tilde B}\xrightarrow{{\tilde \phi}} {\tilde A}$  given  by $T'\to (t_0,\ldots,t_m)$ and $Z_{1j}\to z_{1j}$ modulo $d^3$.

Moreover, we may assume that the composite map $E\xrightarrow{v} A\to {\tilde A}$ factors through $\tilde \phi$. The map ${\tilde G}_1:={\tilde C}[T',Z_{11},\ldots,Z_{1l_1}]/(cT_0'-t_0'',\ldots, cT'_m-t''_m)\xrightarrow{{\tilde w}_1} {\tilde A}$ given by $T'\to (t_0,\ldots,t_m)$, $Z_{1j}\to z_{1j}$ modulo $d^3$
satisfies ${\tilde w}_1(H_{{\tilde G}_1/{\tilde {\Phi}}})\not \subset  \cup_{i=1}^{l_i}(z_{1i},\gamma){\tilde A}$. Clearly, 
$${\tilde B}= {\tilde G}_1/(p-T_0'Z_{11}^{\alpha_{11}}\cdots Z_{1l_1}^{\alpha_{1l_1}}-\sum_{j\in [m]}T'_j\Gamma_j)).$$

 We show by induction on $n$, that there exist an essentially smooth local $\Phi$-algebra $(D,\qq)$ with $p$  from a system of  regular parameters modulo $(\Gamma)$ of $D/\Gamma D$ and $b\in \qq^2$ such that $v$ factors through a smooth $D/(p-b)$-algebra $D'$. This is enough because $D/(p-b,\Gamma)$ is    still regular and so $D'/\Gamma D'$ is too. The case $n=1$ is done in Theorem \ref{t2}.

\vskip 0.5 cm
{\bf Case $n=2$}

We may suppose that $x_2$ is regular in $\tilde{A}$. If  ${\tilde w}_1'(H_{{\tilde G}_1/{\tilde{\Phi} }}){\tilde A}={\tilde A}$ then we may change ${\tilde G}_1$ to be essentially smooth over $\tilde{\Phi}$ (see e.g. \cite[Lemma 2.4]{P0}). Thus we may consider that  ${\tilde G}_1/\Gamma {\tilde G}_1$ and hence also 
${\tilde B}/\Gamma {\tilde B}$ are regular local rings which ends the proof using Proposition \ref{p2} as in Theorem \ref{t2}. 

Otherwise,
the ideal ${\tilde w}_1'(H_{{\tilde G}_1/{\tilde{\Phi} }}){\tilde A}$ is $\mm$-primary and a power $x_2^e$ of $x_2$ is in\\
 ${\tilde w}_1'(H_{{\tilde G}_1/{\tilde{\Phi} }}){\tilde A}$. Moreover as above we may assume that $x_2^e\equiv \sum_i M'_i N'_i$ modulo $I'$ for some $M'_i,N'_i, I'$ associated to ${\tilde G}_1$, 
which has the form  ${{\tilde{ \Phi}}}[Y']/{\tilde I}$, $Y'=(Y'_{1},\ldots,Y'_{m'})$.

Note that  $\bar A:=A/(x_2^{3e})$ is a flat $\bar{\Phi}=\Phi[X_2]/(X_2^{3e})$-algebra, $X_2\to x_2$.   By Theorem \ref{t2} applied to $(\Gamma,X_2)$, $(s,3e)$ we see that   $\bar A:=A/(x_2^{3e})$ is a filtered inductive limit of some Noetherian local $\bar {\Phi}$-algebras of type $ G/(p-b)$, where $ (G,\aa)$ is essentially smooth over  $\bar{\Phi}$, $p$ is from   a regular  system of parameters   modulo $(\Gamma,X_2)$ of $G/(\Gamma ,X_2)G$ and $b\in \aa^2  $. Let $\bar{\omega}:G/(p-b)\to \bar{A}$ be the limit map. We may suppose that the composite map ${\tilde G}_1\xrightarrow{{\tilde w}_1} {\tilde A}\to {\tilde A}/(x_2^{3e})={\bar A}/(d^3)$ factors through ${\tilde \Phi}\otimes_{\Phi}{\bar \omega}$.  

We claim that there exists a Noetherian local $\Phi$-algebra $D$ such that $D/(\Gamma)$ is regular and  $\bar \omega $ factors through $D/X_2^{3e}D$. Indeed,  $G$ is a localization of a $\bar \Phi$-algebra of type $({\bar \Phi}[W]/(g))_{\partial g/\partial W_1}$ for some variables $W=(W_1,\ldots,W_{\lambda})$ and $g$ a monic polynomial in $W_1$ since $G$ is essentially smooth over $\bar \Phi$. Take $\epsilon$   in $A$ lifting ${\bar \omega }(W)$ and $\delta,\delta'\in A$  such that $g(\epsilon)=x_2^{3e}\delta$, $(p-b)(\epsilon)=x_2^{3e}\delta'$ and  let $D$ be the corresponding  localization of 
$$(\Phi[X_2,W,\Delta, \Delta']/(g-X_2^{3e}\Delta, p-b-X_2^{3e}\Delta'))_{\partial g/\partial W_1},$$
 $\Delta,\Delta '$ being new variables. Let $\rho:D\to A$ be the map given  by $W\to g(W)$, $\Delta\to \delta$, $\Delta'\to \delta'$. Then $\bar{\omega} $ factors through $D/X_2^{3e}D\otimes \rho$. Since $({\bar \Phi}[W,\Delta]/(g-X_2^{3e}\Delta))_{\partial g/\partial W_1}$ is smooth over $\bar \Phi$, the corresponding localization $C$ of $(\Phi[X_2,W,\Delta]/(g-X_2^{3e}\Delta))_{\partial g/\partial W_1}$ is such that $C/\Gamma C$ is regular and $p$ is from a regular local system of parameters of $C/\Gamma C$. It follows that $D/\Gamma D$ is  regular local, which shows our claim.

Using Proposition
\ref{p2} for ${\tilde \Phi}[X_2]\to {\tilde A}$, $X_2\to x_2$,  ${\tilde \Phi}[X_2]\otimes_{\tilde \Phi}{\tilde G}_1$, ${\tilde D}=D/d^3D$ and ${\tilde {\Phi}}\otimes_{\Phi}\rho$ we see that ${\tilde w}_1$ factors through an essentially smooth $\tilde D$-algebra  ${\tilde D}'$. Then we may find an essentially smooth 
$D$-algebra $D'$ such that ${\tilde {\Phi}}\otimes_{\Phi}D'\cong {\tilde D}'$. 
Clearly, $D'/\Gamma D'$ is regular because $D/\Gamma D$ is. The map ${\tilde D}'\to {\tilde A} $ lifts by the Implicit Function Theorem to a $D$-morphism $\rho' :D'\to A$ such that $\tilde{\Phi}\otimes_{\Phi} v$ factors through
$\rho'$ modulo $d^3$. Applying again Proposition \ref{p2} for ${\Phi}\to A$, $E,D',\rho'$ we see that $v$ factors through an essentially smooth $D'$-algebra $D''$, such that $D''/\Gamma D''$  is regular since  $D'/\Gamma D'$  is.

\vskip 0.5 cm
{\bf Case $n>2$}

As above we may suppose that  ${\tilde w}_1'(H_{{\tilde G}_1/{\tilde{\Phi} }}){\tilde A}\subset \mm{\tilde A}$.
We will show that given a prime ideal $q\in \Spec A$ such that $q{\tilde A}$ is a minimal prime  ideal of  ${\tilde w}_1'(H_{{\tilde G}_1/{\tilde{\Phi} }})$, the map ${\tilde w}_1'$ factors through a  finite type $\tilde{\Phi}$-algebra of the form ${\tilde P}/(p-b)$, $b\in {\tilde P}$, let us say  ${\tilde w}_1'$ is the composite map 
 ${\tilde G}_1 \to {\tilde P}/(p-b)\xrightarrow{{\tilde\mu}} {\tilde A}'$, where ${\tilde A}'$ is a factor of $\tilde A$, ${\tilde \mu}$ is induced by a map ${\tilde \mu}':{\tilde P}\to {\tilde A}' $ with
${\tilde w}_1'(H_{{\tilde G}_1/{\tilde{\Phi} }}){\tilde A}'\subset \sqrt{{\tilde \mu}(H_{{\tilde P}/{\tilde{\Phi} }}){\tilde A}'}\not \subset q{\tilde A}'$.

For the beginning we assume that height $q=2$. Then $qA_q=(z_1,z_2,\gamma) A_q$ for some $z_1,z_2\in q$ because $R_q$ is a regular local ring of dimension $2$. Assume that $(p,z_2,\gamma)A_q$ is a $qA_q$-primary ideal. Let $\Phi_2 =\Phi[Z_2]_{(p,Z_2,\Gamma)}$ and
  $\Lambda_2$ be the (unique) Cohen $\Phi_2$-algebra  with the residue field  Fr$(A/q)$, that is  $\Lambda_2$ is complete  and the map $\Phi_2\to \Lambda_2$ is flat (even regular). Then the completion ${\hat R}_q$ of the regular local ring  $R_q$ is a finite free module over
$\Lambda_2 /(\Gamma)$. We have $\lambda_2z_2^{e_2}\in {\tilde w}_1'(H_{{\tilde G}_1/{\tilde{\Phi} }}){\tilde A}$ for some $\lambda_2\in A\setminus q$ and $e_2\in {\bf N}$. Changing $z_2$ by $\lambda_2 z_2$ we may suppose that $\lambda_2=1$. Set ${\tilde\Phi}_2=\Phi_2/(d^3)$. Changing 
${\tilde G}_1$ by ${\tilde \Phi}_2\otimes_{{\tilde\Phi}}{\tilde G}_1$ and ${\tilde w}_1' $ canonically we may suppose that ${\tilde G}_1$ is a ${\tilde\Phi}_2$-algebra.

Similar to Case $n=2$  we may assume that ${\tilde G}_1\cong {\tilde{\Phi}_2}[Y']/{\tilde I}$
and that $z_2^{e_2}\equiv \sum_jM_jL_j$  modulo ${\tilde I}$ for a system of polynomials $f^{(2)}$ from ${\tilde I}$, some polynomials $L_j\in ((f^{(2)}): {\tilde I})$ and some minors
$M_j$ of $(\partial f^{(2)}/\partial Y')$. Set ${\tilde \Lambda_2}=\Lambda_2/(d^3)$. Note that  ${\tilde A}_q/(z_2^{3e_2},\gamma)\cong  \hat{A_q}/(d^3,z_2^{3e_2},\gamma)$ is finite free over 
$\tilde{\Lambda}_2/(Z_2^{3e_2},\Gamma)$ with the basis $\{z_1^iz_2^j: (i,j)\in N_2\} $ for some $N_2\subset [0,3e_1]\times [0,3e_2]$ and $e_1\in {\bf N}$ (we may suppose $p\not \in (z_1^{e_1},z_2^{e_2})$ increasing $e_1$ and $e_2$ changing $L_i$ if necessary). It follows that ${\tilde A}_q/(z_2^{3e_2})$  is finite free over 
$\tilde{\Lambda}_2/(Z_2^{3e_2})$ with the same basis.
We have 
$$p\equiv \sum_{(i,j)\in N_2' } t_{20ij}z_1^iz_2^j+\sum_{\lambda\in [m], (i,j)\in N_2''} t_{2\lambda ij}z_1^iz_2^j \gamma_{\lambda}$$
 modulo $(d^3,z_2^{3e_2})$ for some  $t_{20ij}\in A\setminus q$,  $t_{2\lambda ij}\in A$ and some subsets $N_2',N_2''\subset [0,3e_1]\times [0,3e_2]$. 
$\Lambda_2$ is a filtered inductive limit of  smooth ${\Phi}_2$-algebras $C_2$ (see Theorem \ref{gnd}) and let  ${\tilde \psi}_2:{\tilde C}_2/(Z_2^{3e_2})\to {\tilde \Lambda}_2/(Z_2^{3e_2})$, ${\tilde C}_2:=C_2/(d^3)$ be the limit map. We may choose  ${\tilde C}_2$ such that the image of the  map  
$${\tilde \psi}_2':{\tilde C}_2[Z_1]/(Z_1^{3e_1},Z_2^{3e_2})\to {\tilde A}_q/(z_1^{3e_1},z_2^{3e_2}),$$
 $Z_1\to z_1$
 contains $( t_{2\lambda ij})$, let us say $t_{2\lambda ij}={\tilde \psi}_2'(t_{2\lambda ij}')$ for some $t'_{2\lambda ij}\in {\tilde C}_2[Z_1]$, $0\leq \lambda\leq m$. Then 
 $$t_2':=\sum_{(i,j)\in N_2'}t_{20 ij}'Z_1^iZ_2^j+\sum_{(i,j)\in N_2'',  \lambda\in [m]}t_{2\lambda ij}'Z_1^iZ_2^j\Gamma_{\lambda}$$
   is mapped in $p$ modulo  $(d^3,z_2^{3se_2})$ by ${\tilde{\psi}_2'}$.

 Note that ${\tilde C}_2$ is  a localization of a finite type $\tilde{\Phi}_2$-algebra ${\tilde C}_2'$ with ${\tilde \psi}_2'({\tilde C}_2')\subset {\tilde A}/(z_2^{e_2})$ and ${\tilde \psi}_2'(H_{{\tilde C}_2'/{\tilde{\Phi}_2}})\not\subset q{\tilde A}/(z_2^{3e_2})$. 
 Usually we have $t_{2\lambda ij}'\not \in {\tilde C}_2'[Z_1]$ but there exist $c, t''_{2\lambda ij}\in {\tilde C}_2'[Z_1]$ such that ${\tilde \psi}_2'(c)\in {\tilde A}/(z_1^{3e_1},z_2^{3e_2})\setminus q{\tilde A}/(z_1^{3se_1},z_2^{3e_2})$ and $ct'_{2\lambda ij}=t''_{2\lambda ij}$. 
 Since ${\tilde A}_q/(z_1^{3e_1},z_2^{3e_2})$ is a filtered inductive limit of $\Phi_2$-algebras of type ${\tilde B}_2={\tilde C}_2[Z_1]/(Z_1^{3e_1},Z_2^{3e_2},p-t'_2)$ we see that the composite map ${\tilde G}_1\xrightarrow{{\tilde w}'_1} {\tilde A}\to  {\tilde A}_q/(z_1^{3e_1},z_2^{3e_2})$ factors to a such  $\Phi_2$-algebra  ${\tilde B}_2$. 
 
  Set ${\tilde P}_2:={\tilde C}_2'[Z_1,T']/(cT'_{(2\lambda ij}-t''_{20\lambda ij})_{\lambda i j},Z_1^{3e_1},Z_2^{3e_2})$ for some new variables  
 $T'=(T_{2\lambda ij}')_{\lambda i j}$. Let ${\tilde \mu}_2':{\tilde P}_2\to {\tilde A}/(z_1^{3e_1},z_2^{3e_2})$ given by $T'\to ( t_{2\lambda ij})$. Note that ${\tilde \mu}_2'(H_{{\tilde P}_2/{\tilde \Phi}_2})\not \subset q{\tilde A}/(z_1^{3e_1},z_2^{3e_2})$. 
 We may assume that the composite map ${\tilde G}_1\xrightarrow{{\tilde w}'_1} {\tilde A}\to  {\tilde A}/(z_1^{3e_1},z_2^{3e_2})$ factors to a   $\Phi_2$-algebra 
  $${\tilde B}_2'={\tilde P}_2/(p-\sum_{(i,j)\in N_2'}T_{20 ij}'Z_1^iZ_2^j+\sum_{(i,j)\in N_2'', 1\leq \lambda\leq m}T_{2\lambda ij}'Z_1^iZ_2^j\Gamma_{\lambda}).$$

Step by step, using this procedure we may find some elements $(z_i)_{2\leq i\leq \nu}$, some numbers $(e_i)_i$, some finite type $\tilde{\Phi}_i={\tilde\Phi}_2[Z_3,\ldots,Z_i]/(Z_2^{3e_2},\ldots,Z_i^{3e_i}) $-algebras ${\tilde P}_i$,  some morphisms ${\tilde w}'_i:{\tilde P}_i\to A_i:={\tilde A}/(z_3^{3e_2},\ldots,z_i^{3e_i})$ and some elements $b_i\in {\tilde P}_i$, $ 3\leq i\leq \nu$ such that for all  $ 3\leq i\leq \nu$
\begin{enumerate}
\item the composite map ${ \tilde P}_{i-1}\xrightarrow{{\tilde w}'_{i-1}} A_{i-1}\to A_i$ factors by ${\tilde P}_i/(p-b_i)$ namely ${\tilde w}'_{i-1}$ is the composite 
map\\
 ${\tilde P}_{i-1}\to {\tilde P}_i/(p-b_i)\xrightarrow{{\tilde w}_i} A_i$, 
where ${\tilde w}_i$ is induced by ${\tilde w}'_i$,

\item there exists a strict inclusion $\sqrt{{\tilde w}'_{i-1}(H_{{\tilde P}_{i-1}/{\tilde \Phi}_{i-1}})A_i}\subset \sqrt{{\tilde w}'_i(H_{{\tilde P}_i/{\tilde {\Phi}_i}}A_i)}$,
\item $(0:_{A_{i-1}}z_i^{e_i})=(0:_{A_{i-1}}z_i^{e_i+1})$ and $z_i^{e_i}\in {\tilde w}'_{i-1}(H_{{\tilde P}_{i-1}/{\tilde {\Phi}}_{i-1}})A_i$.
\end{enumerate}

By Noetherian induction we may assume that ${\tilde w}'_{\nu}(H_{{\tilde P}_{\nu}/{\tilde {\Phi}}_{\nu}}){\tilde A}_{\nu}$
is $\mm {\tilde A}$-primary. We may choose  $x_n$ such that it is regular in $(R_i)_{i\in [\nu]}$ and  a power $x_n^{e'}$ of $x_n$ is in ${\tilde w}'_{\nu}(H_{{\tilde P}_{\nu}/{\tilde {\Phi}}_{\nu}}){\tilde A}_{\nu}$ (we can assume that
$x_n^{e'}\in \Delta_f((f):J)$ for some $f, J$ defining $P_{\nu}$ as usual). 

Induct on $n>1$. Using the induction hypothesis,
as in Case $n=2$,  $A/(x_n^{3e'})$ is a filtered inductive limit of some Noetherian local ${\tilde \Phi}'={\tilde \Phi}[X_n]/(X_n^{3e'})$-algebras of type $G/(p-b)$, where $(G,\aa)$ is essentially smooth over ${\tilde \Phi}'$, $p$ is from a system of regular parameters of $G/(\Gamma,X_n)$
and $b\in \aa^2$. Let ${\bar \omega}:G/(p-b)\to A/(x_n^{3e'})$ be the limit map. Then we find as in Case $n=2$ a Noetherian local $\Phi$-algebra $D$ such that
$D/(X_n^{3e'})\cong /(p-b)$,  $D/(\Gamma)$ is regular and the composite map  ${\tilde P}_{\nu}\xrightarrow{{\tilde w}'_{\nu}} A_{\nu}\to A_{\nu}/(x_n^{3e'})$ factors through 
${\tilde \Phi}_{\nu}\otimes_{\tilde \Phi}{\bar \omega}$. As in Case $n=2$ using again Proposition \ref{p2} we find a Noetherian $\Phi$-algebra $D'$ and a map $\rho':D'\to A$ such that $D'/(\Gamma)D'$ is regular and ${\tilde w}'_{\nu}$ factors through ${\tilde \Phi}_{\nu}\otimes_{\Phi} D'$. 

In particular, the composite map ${\tilde P}_{\nu-1}\xrightarrow{{\tilde w}'_{\nu-1}} A_{\nu-1}\to A_{\nu}$ factors through ${\tilde  \Phi}_{\nu}\otimes_{\phi} D'$ and the procedure continues. By recurrence  using Proposition \ref{p2} and Remark \ref{r} we find a Noetherian local $\Phi$-algebra $D''$ essentially smooth over $D$ and a map $\rho'':D''\to A$ such that $D''/\Gamma D''$ is regular and ${\tilde w}_1'$ factors through ${\tilde \Phi}\otimes_{\Phi} D''$. Finally the same procedure shows that $v$ factors trough a $\Phi$-algebra $F$ essentially smooth over $D$, with $F/\Gamma F$ regular.

\hfill\ \end{proof}

\begin{Remark} \label{r1}{\em To show the BQ Conjecture for a regular local ring of unequal characteristic applying Theorem \ref{s} we need to prove the following facts:

i) if the BQ Conjecture holds for $\Lambda'_{i-1}$ then it holds for $\Lambda_i$, $1\leq i\leq r$,

ii) if the BQ Conjecture holds for $\Lambda_i$ then it holds for $\Lambda'_i$, $1\leq i<r$.

The first fact could be proved using the ideas of Lindel and Swan. For the second fact we have to show the following result

iii) if the BQ Conjecture holds for a regular local ring $A$ and $b\in A$ belongs to a regular system of parameters of $A$  then it holds for  a localization  of $A[W]/(W^e-b)$, $e\in {\bf N}$ (in fact for  $A[W]/(W^e-b)$ using \cite{R}).

This might be wrong as an example from \cite{La} suggests. If $A={\bf R}[X_1]_{(X_1)}$, $B=(A[X_2]/ (X_2^3-X_1^2)))_{(X_1,X_2)}$ then there exist projective modules over $B[T]$ of rank one, which are not free. On the other hand, Vorst shows in \cite{V} that if the BQ Conjecture holds for a regular local ring $A$ then it holds for a factor of a polynomial $A$-algebra by a monomial ideal.}
\end{Remark} 

\begin{Corollary}\label{c}
All regular local rings are filtered inductive limits of excellent regular local rings.
\end{Corollary}

\vskip 0.5 cm

\section{Appendix by K\k{e}stutis \v{C}esnavi\v{c}ius\protect\footnote{Laboratoire de Math\'{e}matiques d'Orsay, Univ.~Paris-Sud, CNRS, Universit\'{e} Paris-Saclay, 91405 Orsay, France; email: {\sf kestutis@math.u-psud.fr}}: A purity conjecture for flat cohomology}

The following purity conjecture seems to be a part of the folklore.

\begin{Conjecture}
For a regular local ring $(R, \mm)$ of dimension $d$ and a commutative, finite, flat $R$-group scheme $G$, the fppf cohomology with supports vanishes in low degrees as follows
$$
H^i_{\mm}(R, G) = 0 \ \  \mbox{for} \ \ i < d;
$$
equivalently, setting $U_R := \Spec(R) \setminus \{\mm\}$, the pullback map 
$$
H^i(R, G) \to H^i(U_R, G) \ \ \mbox{is an isomorphism for} \ \ i < d - 1 \ \ \mbox{and injective for} \ \ i = d - 1.
$$
\end{Conjecture}


As usual, we form the fppf cohomology of a scheme $S$ in the site whose objects are all the $S$-schemes and coverings are jointly surjective morphisms that are flat and locally of finite presentation.
\medskip

{\bf Examples of known cases}
\begin{enumerate}
\item
In its equivalent form that involves $U_R$, the conjecture is known for cohomological degrees $i = 0$ and $i = 1$. More precisely, the map
$$
\ \  H^0(R, G) \to H^0(U_R, G) \ \ \text{is bijective when} \ \  d > 0
$$
by a well-known lemma on sections of finite schemes over normal bases (for the lack of a better reference, see \cite[Lemma 3.1.9]{modular-description}). This lemma and the representability of $G$-torsors also show that the map
$$
\ \  H^1(R, G) \to H^1(U_R, G)
$$
is injective when $d > 0$; by a result of Moret-Bailly \cite[Lemme 2]{MB85} (whose detailed proof may be found, for instance, in \cite[\S3.1]{Mar16}), this map is bijective when $d > 1$. Of course, there is no need to assume that $G$ is commutative for these results in low degrees.

\item
The case when the order of $G$ is invertible in $R$ is also known; in fact, in this case, by the absolute cohomological purity conjecture proved by Gabber, $H^i_{\mm}(R, G) = 0$ for $i < 2d$, see \cite{F}.

\item
We have $H^i_{\mm}(R, \mathbb{G}_a) = 0$ for $i < d$ because the depth of $R$ is $d$ (see \cite[III.3.4]{SGA}), so, in the case when $R$ is an $\bf{F}_p$-algebra, the conjecture holds for $G = \alpha_p$ (where $\alpha_p$ is the kernel of the absolute Frobenius morphism of $\mathbb{G}_a$).

\item
By the purity for the Brauer group, whose proof was finished in \cite{C}, the groups $H^i_{\mm}(R, \mathbb{G}_m)$ vanish when $d \ge 2$ and $i \le 3$; thus, $H^i_{\mm}(R, \mu_p) = 0$ when $i \leq 3$ with $i < d$ (where $\mu_p$ is the kernel of the absolute Frobenius morphism of $\mathbb{G}_m$).

\end{enumerate}

The goal of this appendix is to illustrate the main results of the paper with the following reduction.

\begin{Proposition}\label{pu}
For every regular local ring $(R, \mm)$ that is a filtered direct limit of \emph{excellent} regular local rings, the conjecture above reduces to the case when $R$ is complete.
\end{Proposition}

\begin{proof}
By descending the elements that comprise an $(R/\mm)$-basis of $\mm/\mm^2$ to a finite level in the direct limit, we may assume that $(R, \mm)$ is a filtered direct limit of excellent regular local rings each of whose dimension is at least that of $R$. Thus, the limit formalism reduces us to the case when $R$ is excellent. Excision (see \cite[proof of III.1.27]{Mi}) 
and \cite[18.7.6]{EGA} then reduce us further to the case when $R$ is both excellent and Henselian.  We then use the N\'{e}ron--Popescu desingularization (see Theorem \ref{gnd}) to express the $\mm$-adic completion $(\widehat{R}, \widehat{\mm})$ of $(R, \mm)$ as a filtered direct limit of essentially smooth, local $R$-algebras $(R_j, \mm_j)$. Since $R$ is Henselian local and shares the residue field with $\widehat{R}$, Hensel's lemma \cite[18.5.17]{EGA} ensures that the structure maps $R \to R_j$ have sections $R_j \to R$. In particular, they induce the injections
$$
H^i_{\mm}(R, G) \hookrightarrow H^i_{\mm_j}(R_j, G), \ \ \mbox{which give the injection} \ \ H^i_\mm(R, G) \hookrightarrow H^i_{\widehat{\mm}}(\widehat{R}, G).
$$
This achieves the promised reduction to the case when $R$ is complete.
\end{proof}

\begin{Remark} \label{r2} {\em By Corollary \ref{c}, we see that the above Conjecture is reduced to the complete case.}
\end{Remark}

{\bf Acknowledgements:}

We owe thanks to K\k{e}stutis \v{C}esnavi\v{c}ius for some helpful remarks on Lemmas 5, 9, Proposition 11, Remark 15 and especially for the appendix and  to M. V\^aj\^aitu and A. Zaharescu who hinted us Remark \ref{r0}. Also we owe thanks to Moret-Bailly for some useful remarks on the appendix and to the Referee for many corrections.
This work has been partially elaborated in the frame of the International
Research Network ECO-Math.

\end{document}